\documentclass[12pt,twoside]{article}
\usepackage{amsmath,amssymb}
\usepackage{amsthm}
\usepackage{upref}
\usepackage{graphicx,amssymb}
\usepackage{color}
\numberwithin{equation}{section}

\theoremstyle{plain}
\newtheorem{theorem}{Theorem}[section]
\newtheorem{proposition}[theorem]{Proposition}
\newtheorem{definition}[theorem]{Definition}
\newtheorem{corollary}[theorem]{Corollary}

\theoremstyle{definition}
\newtheorem{remark}[theorem]{Remark}
\newtheorem{problem}[theorem]{Problem}

\begin{document}

\newcommand{\eq}{equation}
\newcommand{\real}{\ensuremath{\mathbb R}}
\newcommand{\comp}{\ensuremath{\mathbb C}}
\newcommand{\rn}{\ensuremath{{\mathbb R}^n}}
\newcommand{\tn}{\ensuremath{{\mathbb T}^n}}
\newcommand{\rnp}{\ensuremath{\real^n_+}}
\newcommand{\rnn}{\ensuremath{\real^n_-}}
\newcommand{\Rn}{\ensuremath{{\mathbb R}^{n-1}}}
\newcommand{\Zn}{\ensuremath{{\mathbb Z}^{n-1}}}
\newcommand{\no}{\ensuremath{\nat_0}}
\newcommand{\ganz}{\ensuremath{\mathbb Z}}
\newcommand{\zn}{\ensuremath{{\mathbb Z}^n}}
\newcommand{\zom}{\ensuremath{{\mathbb Z}_{\Om}}}
\newcommand{\zOm}{\ensuremath{{\mathbb Z}^{\Om}}}
\newcommand{\As}{\ensuremath{A^s_{p,q}}}
\newcommand{\Bs}{\ensuremath{B^s_{p,q}}}
\newcommand{\Fs}{\ensuremath{F^s_{p,q}}}
\newcommand{\Fsr}{\ensuremath{F^{s,\rloc}_{p,q}}}
\newcommand{\nat}{\ensuremath{\mathbb N}}
\newcommand{\Om}{\ensuremath{\Omega}}
\newcommand{\di}{\ensuremath{{\mathrm d}}}
\newcommand{\sn}{\ensuremath{{\mathbb S}^{n-1}}}
\newcommand{\Ac}{\ensuremath{\mathcal A}}
\newcommand{\Acs}{\ensuremath{\Ac^s_{p,q}}}
\newcommand{\Bc}{\ensuremath{\mathcal B}}
\newcommand{\Cc}{\ensuremath{\mathcal C}}
\newcommand{\cc}{{\scriptsize $\Cc$}${}^s (\rn)$}
\newcommand{\ccd}{{\scriptsize $\Cc$}${}^s (\rn, \delta)$}
\newcommand{\Fc}{\ensuremath{\mathcal F}}
\newcommand{\Lc}{\ensuremath{\mathcal L}}
\newcommand{\Mc}{\ensuremath{\mathcal M}}
\newcommand{\Ec}{\ensuremath{\mathcal E}}
\newcommand{\Pc}{\ensuremath{\mathcal P}}
\newcommand{\Efr}{\ensuremath{\mathfrak E}}
\newcommand{\Mfr}{\ensuremath{\mathfrak M}}
\newcommand{\Mbf}{\ensuremath{\mathbf M}}
\newcommand{\Dbb}{\ensuremath{\mathbb D}}
\newcommand{\Lbb}{\ensuremath{\mathbb L}}
\newcommand{\Pbb}{\ensuremath{\mathbb P}}
\newcommand{\Qbb}{\ensuremath{\mathbb Q}}
\newcommand{\Rbb}{\ensuremath{\mathbb R}}
\newcommand{\vp}{\ensuremath{\varphi}}
\newcommand{\hra}{\ensuremath{\hookrightarrow}}
\newcommand{\supp}{\ensuremath{\mathrm{supp \,}}}
\newcommand{\ssupp}{\ensuremath{\mathrm{sing \ supp\,}}}
\newcommand{\dist}{\ensuremath{\mathrm{dist \,}}}
\newcommand{\unif}{\ensuremath{\mathrm{unif}}}
\newcommand{\ve}{\ensuremath{\varepsilon}}
\newcommand{\vk}{\ensuremath{\varkappa}}
\newcommand{\vr}{\ensuremath{\varrho}}
\newcommand{\pa}{\ensuremath{\partial}}
\newcommand{\oa}{\ensuremath{\overline{a}}}
\newcommand{\ob}{\ensuremath{\overline{b}}}
\newcommand{\of}{\ensuremath{\overline{f}}}
\newcommand{\LA}{\ensuremath{L^r\!\As}}
\newcommand{\LcA}{\ensuremath{\Lc^{r}\!A^s_{p,q}}}
\newcommand{\LcdA}{\ensuremath{\Lc^{r}\!A^{s+d}_{p,q}}}
\newcommand{\LcB}{\ensuremath{\Lc^{r}\!B^s_{p,q}}}
\newcommand{\LcF}{\ensuremath{\Lc^{r}\!F^s_{p,q}}}
\newcommand{\Lf}{\ensuremath{L^r\!f^s_{p,q}}}
\newcommand{\La}{\ensuremath{\Lambda}}
\newcommand{\Lob}{\ensuremath{L^r \ob{}^s_{p,q}}}
\newcommand{\Lof}{\ensuremath{L^r \of{}^s_{p,q}}}
\newcommand{\Loa}{\ensuremath{L^r\, \oa{}^s_{p,q}}}
\newcommand{\Lcoa}{\ensuremath{\Lc^{r}\oa{}^s_{p,q}}}
\newcommand{\Lcob}{\ensuremath{\Lc^{r}\ob{}^s_{p,q}}}
\newcommand{\Lcof}{\ensuremath{\Lc^{r}\of{}^s_{p,q}}}
\newcommand{\Lca}{\ensuremath{\Lc^{r}\!a^s_{p,q}}}
\newcommand{\Lcb}{\ensuremath{\Lc^{r}\!b^s_{p,q}}}
\newcommand{\Lcf}{\ensuremath{\Lc^{r}\!f^s_{p,q}}}
\newcommand{\id}{\ensuremath{\mathrm{id}}}
\newcommand{\tr}{\ensuremath{\mathrm{tr\,}}}
\newcommand{\trd}{\ensuremath{\mathrm{tr}_d}}
\newcommand{\trL}{\ensuremath{\mathrm{tr}_L}}
\newcommand{\ext}{\ensuremath{\mathrm{ext}}}
\newcommand{\re}{\ensuremath{\mathrm{re\,}}}
\newcommand{\Rea}{\ensuremath{\mathrm{Re\,}}}
\newcommand{\Ima}{\ensuremath{\mathrm{Im\,}}}
\newcommand{\loc}{\ensuremath{\mathrm{loc}}}
\newcommand{\rloc}{\ensuremath{\mathrm{rloc}}}
\newcommand{\osc}{\ensuremath{\mathrm{osc}}}
\newcommand{\pr}{\pageref}
\newcommand{\wh}{\ensuremath{\widehat}}
\newcommand{\wt}{\ensuremath{\widetilde}}
\newcommand{\ol}{\ensuremath{\overline}}
\newcommand{\os}{\ensuremath{\overset}}
\newcommand{\Li}{\ensuremath{\overset{\circ}{L}}}
\newcommand{\Ai}{\ensuremath{\os{\, \circ}{A}}}
\newcommand{\Ci}{\ensuremath{\os{\circ}{\Cc}}}
\newcommand{\dom}{\ensuremath{\mathrm{dom \,}}}
\newcommand{\SA}{\ensuremath{S^r_{p,q} A}}
\newcommand{\SB}{\ensuremath{S^r_{p,q} B}}
\newcommand{\SF}{\ensuremath{S^r_{p,q} F}}
\newcommand{\Hc}{\ensuremath{\mathcal H}}
\newcommand{\Nc}{\ensuremath{\mathcal N}}
\newcommand{\Lci}{\ensuremath{\overset{\circ}{\Lc}}}
\newcommand{\bmo}{\ensuremath{\mathrm{bmo}}}
\newcommand{\BMO}{\ensuremath{\mathrm{BMO}}}
\newcommand{\cm}{\\[0.1cm]}
\newcommand{\Aa}{\ensuremath{\os{\, \ast}{A}}}
\newcommand{\Ba}{\ensuremath{\os{\, \ast}{B}}}
\newcommand{\Fa}{\ensuremath{\os{\, \ast}{F}}}
\newcommand{\Aas}{\ensuremath{\Aa{}^s_{p,q}}}
\newcommand{\Bas}{\ensuremath{\Ba{}^s_{p,q}}}
\newcommand{\Fas}{\ensuremath{\Fa{}^s_{p,q}}}
\newcommand{\Ca}{\ensuremath{\os{\, \ast}{{\mathcal C}}}}
\newcommand{\Cas}{\ensuremath{\Ca{}^s}}
\newcommand{\Car}{\ensuremath{\Ca{}^r}}
\newcommand{\bl}{$\blacksquare$}

\begin{center}
{\Large Mapping properties of Fourier transforms}
\\[1cm]
{Hans Triebel}
\\[0.2cm]
Institut f\"{u}r Mathematik\\
Friedrich--Schiller--Universit\"{a}t Jena\\
07737 Jena, Germany
\\[0.1cm]
email: hans.triebel@uni-jena.de
\end{center}

\begin{abstract}
The paper deals with continuous and compact mappings generated by the Fourier transform between distinguished function spaces on \rn.
The degree of compactness will be measured in terms of related entropy numbers. We are more interested in the interplay of already 
available ingredients than in generality.
\end{abstract}



\section{Introduction}  \label{S1}
Let $F$,
\begin{\eq}   \label{1.1}
\big(F \vp\big)(\xi) = (2 \pi)^{-n/2} \int_{\rn} e^{-i x \xi} \, \vp (x) \, \di x, \qquad \vp \in S(\rn), \quad \xi \in \rn,
\end{\eq}
be the classical Fourier transform, extended  in the usual way to $S'(\rn)$, $n\in \nat$. The mapping properties
\begin{\eq}   \label{1.2}
FS(\rn) = S(\rn), \qquad FS'(\rn) = S'(\rn),
\end{\eq}
and
\begin{\eq}   \label{1.3}
F: \ \ L_p (\rn) \hra L_{p'} (\rn), \quad  1\le p \le 2, \quad \frac{1}{p} + \frac{1}{p'} =1, \quad FL_2 (\rn) = L_2 (\rn),
\end{\eq}
are cornerstones of Fourier Analysis. Let, as illustrated in Figure 1 below,
\begin{\eq}   \label{1.4}
X^s_p (\rn) =
\begin{cases}
L_p (\rn) &\text{if $2\le p<\infty$, $s=0$}, \\
B^s_p (\rn) & \text{if $2 \le p <\infty$, $s>0$}, \\
B^s_p (\rn) &\text{if $1<p \le 2$, $s \ge d^n_p$},
\end{cases}
\end{\eq}
and
\begin{\eq}   \label{1.5}
Y^s_p (\rn) =
\begin{cases}
B^s_p (\rn) &\text{if $2 \le p <\infty$, $s \le d^n_p$}, \\
B^s_p (\rn) &\text{if $1<p \le 2$, $s<0$}, \\
L_p (\rn) &\text{if $1<p\le 2$, $s=0$},
\end{cases}
\end{\eq}
where
\begin{\eq}   \label{1.6}
B^s_p (\rn) = B^s_{p,p} (\rn), \quad s\in \real, \quad 0<p\le \infty, \quad B^s_\infty (\rn) = \Cc^s (\rn),
\end{\eq}
are distinguished Besov spaces (H\"{o}lder--Zygmund spaces if $p=\infty$) and
\begin{\eq}   \label{1.7}
d^n_p =2n \big( \frac{1}{p} - \frac{1}{2} \big), \qquad 0<p\le \infty, \quad n\in \nat.
\end{\eq}
It is our first aim to justify the continuous mapping
\begin{\eq}  \label{1.8}
F: \quad X^{s_1}_p (\rn) \hra Y^{s_2}_p (\rn)
\end{\eq}
and that for fixed $p$ the restrictions for $s_1$ and $s_2$ are natural (and not an artefact of the underlying methods). In particular
this mapping is compact if, and only if, neither $X^{s_1}_p (\rn)$ is limiting (this means $s_1=0$ for $2\le p<\infty$ or $s_1 = d^n_p$ for
$1<p\le 2$) nor $Y^{s_2}_p (\rn)$ is limiting (this means $s_2 = d^n_p$ for $2\le p <\infty$ or $s_2 =0$ for $1<p \le 2$). It is our 
second aim to discuss the degree of compactness in terms of related entropy numbers. 

Section \ref{S2} deals with definitions and some basic assertions. The outcome in Theorem \ref{T2.5} might be of some self--contained
interest. But otherwise it serves mainly as a tool for what follows afterwards. Continuous mappings of type \eqref{1.8} are
considered  in Section \ref{S3}. Related compactness properties are treated in Section \ref{S4}. The Theorems \ref{T3.2} and \ref{T4.8}
are the main assertions of this paper. In Section \ref{S5} we collect some problems which might be of interest for future research.

In a further paper  we combine what has been obtained so far with (more or less known) mapping properties of pseudo--differential 
operators of H\"{o}rmander type $S^\sigma_{1, \delta} (\rn)$, $0\le \delta \le 1$, resulting in a spectral theory for related compact
operators. Otherwise we are not aware  whether mappings of $F$ between (unweighted) function spaces beyond the very classical 
assertions in \eqref{1.3} have already been studied in the literature. The only exception known to us is \cite{Ryd20} which will be
commented below. If weights are admitted, especially of power type $|x|^\lambda$, $x\in \rn$, $\lambda \in \real$, then continuous 
mappings of $F$ between weighted $L_p$--spaces, are a matter of so--called Pitt inequalities. A brief survey of what in known including
adequate references, may be found in the recent paper \cite{DoV21} dealing with related vector--valued extensions. But this is not our
topic. However we suggest in Problem \ref{P5.4} to have a closer look at weighted spaces.

\section{Definitions and basic assertions}    \label{S2}
\subsection{Definitions}   \label{S2.1}
We use standard notation. Let $\nat$ be the collection of all natural numbers and $\no = \nat \cup \{0 \}$. Let $\rn$ be Euclidean $n$-space where
$n\in \nat$. Put $\real = \real^1$.
 Let $S(\rn)$ be the Schwartz space of all complex-valued rapidly decreasing infinitely differentiable functions on $\rn$ and let $S' (\rn)$ be the dual space of all tempered distributions on \rn.
Furthermore, $L_p (\rn)$ with $0< p \le \infty$, is the standard complex quasi-Banach space with respect to the Lebesgue measure, quasi-normed by
\begin{\eq}   \label{2.1}
\| f \, | L_p (\rn) \| = \Big( \int_{\rn} |f(x)|^p \, \di x \Big)^{1/p}
\end{\eq}
with the obvious modification if $p=\infty$.  
As usual, $\ganz$ is the collection of all integers; and $\zn$ where $n\in \nat$ denotes the
lattice of all points $m= (m_1, \ldots, m_n) \in \rn$ with $m_k \in \ganz$. 

 If $\vp \in S(\rn)$ then
\begin{\eq}  \label{2.2}
\wh{\vp} (\xi) = (F \vp)(\xi) = (2\pi )^{-n/2} \int_{\rn} e^{-ix \xi} \vp (x) \, \di x, \qquad \xi \in  \rn,
\end{\eq}
denotes the Fourier transform of \vp. As usual, $F^{-1} \vp$ and $\vp^\vee$ stand for the inverse Fourier transform, given by the right-hand side of
\eqref{2.2} with $i$ in place of $-i$. Here $x \xi$ stands for the scalar product in \rn. Both $F$ and $F^{-1}$ are extended to $S'(\rn)$ in the
standard way. Let $\vp_0 \in S(\rn)$ with
\begin{\eq}   \label{2.3}
\vp_0 (x) =1 \ \text{if $|x|\le 1$} \quad \text{and} \quad \vp_0 (x) =0 \ \text{if $|x| \ge 3/2$},
\end{\eq}
and let
\begin{\eq}   \label{2.4}
\vp_k (x) = \vp_0 (2^{-k} x) - \vp_0 (2^{-k+1} x ), \qquad x\in \rn, \quad k\in \nat.
\end{\eq}
Since
\begin{\eq}   \label{2.5}
\sum^\infty_{j=0} \vp_j (x) =1 \qquad \text{for} \quad x\in \rn,
\end{\eq}
$\vp =\{ \vp_j \}^\infty_{j=0}$ forms a dyadic resolution of unity. The entire analytic functions $(\vp_j \wh{f} )^\vee (x)$ make sense pointwise in $\rn$ for any $f\in S'(\rn)$. 

\begin{definition}   \label{D2.1}
Let $\vp = \{ \vp_j \}^\infty_{j=0}$ be the above dyadic resolution  of unity. Let
\begin{\eq}   \label{2.6}
s\in \real \qquad \text{and} \qquad 0<p,q \le \infty.
\end{\eq}
Then $\Bs (\rn)$ is the collection of all $f \in S' (\rn)$ such that
\begin{\eq}   \label{2.7}
\| f \, | \Bs (\rn) \|_{\vp} = \Big( \sum^\infty_{j=0} 2^{jsq} \big\| (\vp_j \widehat{f})^\vee \, | L_p (\rn) \big\|^q \Big)^{1/q}
\end{\eq}
is finite $($with the usual modification if $q= \infty)$. 
\end{definition}

\begin{remark}   \label{R2.2}
These are the well--known inhomogeneous Besov spaces. They are independent of the above resolution of unity $\vp$ according to 
\eqref{2.3}--\eqref{2.5} (equivalent quasi--norms). This justifies the omission of the subscript $\vp$ in \eqref{2.7} in the sequel.
In \cite{T20} one finds (historical) references, explanations and discussions. We do not need their counterparts $\Fs (\rn)$ in this
paper. They will be mentioned occasionally by passing. We rely mainly on the special cases
\begin{\eq}   \label{2.8}
B^s_p (\rn) = B^s_{p,p} (\rn) \qquad 0<p \le \infty, \quad s\in \real,
\end{\eq}
\begin{\eq}   \label{2.9}
\Cc^s (\rn) = B^s_\infty (\rn), \qquad \text{and} \qquad H^s (\rn) = B^s_2 (\rn), \qquad s\in \real.
\end{\eq}
\end{remark}

As already indicated in the Introduction the mapping theory for the Fourier transform as developed in this paper is governed by the
{\em source spaces} $X^s_p (\rn)$ in \eqref{1.4} and the {\em target spaces} $Y^s_p (\rn)$ in \eqref{1.5}. We repeat and complement
these definitions. Let
\begin{\eq}  \label{2.10}
\tau^{n+}_p = \max (0, d^n_p) \qquad \text{and} \qquad \tau^{n-}_p = \min (0, d^n_p)
\end{\eq}
where
\begin{\eq}   \label{2.11}
d^n_p = 2n \big( \frac{1}{p} - \frac{1}{2} \big), \qquad 0<p \le \infty, \quad n\in \nat.
\end{\eq}

\begin{definition}   \label{D2.3}
Let $B^s_p (\rn)$, $n\in \nat$, $1<p<\infty$ and $s\in \real$ be the distinguished  Besov spaces according to Definition \ref{D2.1}
and \eqref{2.8}.
\cm
{\em (i)} Then
\begin{\eq}   \label{2.12}
X^s_p (\rn) =
\begin{cases}
L_p (\rn) &\text{if $2\le p<\infty$, $s=0$}, \\
B^s_p (\rn) & \text{if $2 \le p <\infty$, $s>0$}, \\
B^s_p (\rn) &\text{if $1<p \le 2$, $s \ge d^n_p$},
\end{cases}
\end{\eq}
and
\begin{\eq}   \label{2.13}
Y^s_p (\rn) =
\begin{cases}
B^s_p (\rn) &\text{if $2 \le p <\infty$, $s \le d^n_p$}, \\
B^s_p (\rn) &\text{if $1<p \le 2$, $s<0$}, \\
L_p (\rn) &\text{if $1<p\le 2$, $s=0$}.
\end{cases}
\end{\eq}
{\em (ii)}  The space $X^s_p (\rn)$ is called limiting if $s= \tau^{n+}_p$ and the space $Y^s_p (\rn)$ is called limiting if $s=
\tau^{n-}_p$.
\end{definition}
\begin{figure}[t]
\begin{minipage}{\textwidth}
~\hfill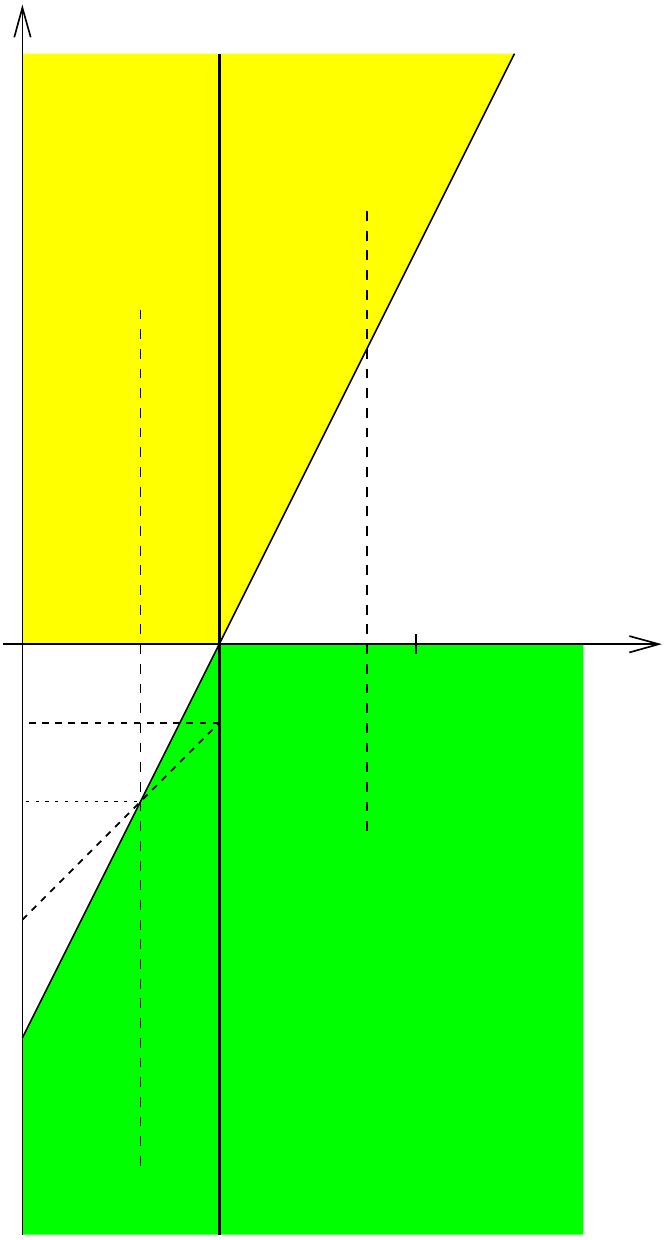\hfill~\\
\end{minipage}
Figure 1: Fourier  mappings \label{Fig1}
\end{figure}

\begin{remark}   \label{R2.4}
Figure \ref{Fig1} illustrates that the source spaces $X^s_p (\rn)$ are defined in the region
\begin{\eq}   \label{2.14}
\big\{ \big( \frac{1}{p}, s \big): \ 1<p<\infty, \ s \ge \tau^{n+}_p \big\}
\end{\eq}
whereas the target spaces $Y^s_p (\rn)$ are defined in the region
\begin{\eq}   \label{2.15}
\big\{ \big( \frac{1}{p}, s \big): \ 1<p<\infty, \ s \le \tau^{n-}_p \big\}.
\end{\eq}
The limiting spaces are located at the boundaries of these regions. They will play a rather peculiar role in what follows.
\end{remark}

\subsection{Basic assertions}    \label{S2.2}
We are mainly interested in continuous and compact mappings of type \eqref{1.8} for the Fourier transform $F$ as recalled in 
\eqref{1.1}. Limiting spaces as introduced in Definition \ref{D2.3} will play a peculiar role. To prepare what follows we adopt first
a wider point of view. Let $\Bs (\rn)$ be the Besov spaces according to Definition \ref{D2.1}, including $B^s_p (\rn)$ as in 
\eqref{2.8} and $\Cc^s (\rn)$ as in \eqref{2.9}. Recall that $F: \ A(\rn) \hra B(\rn)$ means that $f$ (better its restriction from
$S'(\rn)$ to $A(\rn)$) maps $A(\rn)$ continuously into $B(\rn)$ where both $A(\rn)$ and $B(\rn)$ are quasi--Banach spaces, continuously
embedded in $S'(\rn)$. This will also be written as $A(\rn) \os{F}{\hra} B(\rn)$. Similarly, $A(\rn) \os{\id}{\hra} B(\rn)$ stands for 
the usual embedding of $A(\rn)$  into $B(\rn)$. We take for granted that the reader is familiar with classical assertions about 
embeddings between the above function spaces, especially of type $\Bs (\rn)$. In case of doubt one may consult \cite{T20} and the 
detailed references and discussions within. Let again
\begin{\eq}   \label{2.16}
d^n_p = 2n \big( \frac{1}{p} - \frac{1}{2} \big), \qquad 0<p \le \infty, \quad n\in \nat.
\end{\eq}

\begin{theorem}   \label{T2.5}
Let $1\le p \le \infty$ and $\frac{1}{p}+ \frac{1}{p'} =1$. Then the mapping
\begin{\eq}   \label{2.17}
F: \quad L_p (\rn) \hra \Cc^{- \frac{n}{p'}} (\rn)
\end{\eq}
is continuous but not compact. Furthermore,
\begin{\eq}   \label{2.18}
L_p (\rn) \os{F}{\hra} L_{p'} (\rn) \os{\id}{\hra} \Cc^{- \frac{n}{p'}} (\rn) \quad \text{if} \quad 1\le p \le2,
\end{\eq}
and
\begin{\eq}   \label{2.19}
L_p (\rn) \os{F}{\hra} B^{-n (\frac{1}{2} - \frac{1}{p})}_{2,p} (\rn) \os{\id}{\hra} B^{d^n_p}_p (\rn) \os{\id}{\hra}
\Cc^{- \frac{n}{p'}} (\rn) \quad \text{if} \quad 2<p \le \infty.
\end{\eq}
\end{theorem}

\begin{proof}
{\em Step 1.} The well--known embedding  $\id$ in \eqref{2.18} and \eqref{1.3} justify \eqref{2.18}.
\cm
{\em Step 2.} Let $2<p \le \infty$,
\begin{\eq}  \label{2.20}
\frac{1}{2} = \frac{1}{q} + \frac{1}{p} \quad \text{and} \quad \sigma = - n \big( \frac{1}{2} - \frac{1}{p} \big) = - \frac{n}{q}.
\end{\eq}
Then it follows from $\|g^\vee \, | L_2 (\rn) \| = \|g \, |L_2 (\rn) \|$ and H\"{o}lder's inequality applied to \eqref{2.7} that
\begin{\eq}   \label{2.21}
\| Ff \, | B^\sigma_{2,\infty} (\rn) \| \sim \sup_{j\in \no} 2^{j \sigma} \, \| \vp_j f \, | L_2 (\rn) \|  \le c\, \sup_{j \in \no}
2^{j (\sigma + \frac{n}{q})} \, \| f \, | L_p (\rn) \|
\end{\eq}
and
\begin{\eq}   \label{2.22}
F: \quad L_p (\rn) \hra B^{-n (\frac{1}{2} - \frac{1}{p})}_{2, \infty} (\rn).
\end{\eq}
Let $2<p_1 <p<p_2 \le \infty$ and $\frac{1}{p} = \frac{1-\theta}{p_1} + \frac{\theta}{p_2}$. Then one has by real interpolation
according to \cite[Section 1.18.6, pp.\,131--134]{T78} that
\begin{\eq}   \label{2.23}
\big( L_{p_1} (\rn), L_{p_2} (\rn) \big)_{\theta,r} = L_{p,r} (\rn), \qquad 1 \le r \le \infty,
\end{\eq}
where $L_{p,r} (\rn)$ are the well--known Lorentz spaces quasi--normed according to \eqref{2.30} below with $L_{p,p} (\rn) = L_p (\rn)
$. Applied to \eqref{2.22} combined with a related real interpolation of Besov spaces one obtains
\begin{\eq}   \label{2.24}
F: \quad L_{p,r} (\rn) \hra B^{-n (\frac{1}{2} - \frac{1}{p})}_{2,r} (\rn), \qquad 1\le r \le \infty,
\end{\eq}
and for $2<r=p <\infty$ the first mapping in \eqref{2.19}. The case $p=\infty$ is already covered by \eqref{2.22}. The remaining
embeddings in \eqref{2.19} are classical based on $p>2$ and
\begin{\eq}   \label{2.25}
B^{s_1}_{p_1, q_1} (\rn) \hra B^{s_2}_{p_2, q_2} (\rn), \qquad s_1 - \frac{n}{p_1} = s_2 - \frac{n}{p_2}
\end{\eq}
with $0<p_1 \le p_2 \le \infty$ and $0<q_1 \le q_2 \le \infty$.
\cm
{\em Step 3.} It remains to prove that the mapping in \eqref{2.17} is not compact. Let $\psi$ be a non--trivial compactly supported
$C^\infty$ function  in $\rn$ and $\psi_j (x) = \psi(2^{-j} x)$, $j\in \nat$, such that
\begin{\eq}   \label{2.26}
\psi_j (x) \cdot \vp_j (x) = \psi_j (x) \quad \text{and} \quad \supp \psi_j \cap \supp \vp_k = \emptyset \quad \text{if} \quad k\in
\no, \quad j \not= k, 
\end{\eq}
where $\{\vp_j \}$ is a suitably chosen resolution of unity according to \eqref{2.3}--\eqref{2.5}. Let $f_j (x) = 2^{-\frac{jn}{p}}
\psi_j (x)$. Then
\begin{\eq}   \label{2.27}
\| f_j \, | L_p (\rn) \| \sim 1, \qquad j\in \nat,
\end{\eq}
whereas
\begin{\eq}   \label{2.28}
\| \wh{f_j} \, | \Cc^{-\frac{n}{p'}} (\rn) \| \sim \sup_{k\in \no, x\in \rn} 2^{-\frac{n}{p'}k} \big| (\vp_k f_j )^\vee (x) \big| 
\sim \sup_{x\in \rn} 2^{-\frac{jn}{p'} - \frac{jn}{p}} \, |\psi_j^\vee (x) | \sim 1
\end{\eq}
and
\begin{\eq}   \label{2.29}
\| \wh{f_j} - \wh{f_l} \, | \Cc^{-\frac{n}{p'}} (\rn) \| \sim 1 \quad \text{if} \quad j\not= l.
\end{\eq}
This shows that the mapping $F$ in \eqref{2.17} is not compact.
\end{proof}

\begin{remark}   \label{R2.6}
We proved in \eqref{2.24} a little bit more than stated in the above theorem. This can be elaborated as follows. Recall that the 
Lorentz $L_{p,r} (\rn)$ with $0<p<\infty$ and $0<r \le \infty$ are quasi--normed by
\begin{\eq}   \label{2.30}
\| f \, | L_{p,r} (\rn) \| = \Big( \int^\infty_0 \big( t^{1/p} f^* (t) \big)^r \frac{\di t}{t} \Big)^{1/r}
\end{\eq}
(usual modification if $r=\infty$), where $f^* (t)$ is the decreasing rearrangement of $f$. Let $0<\theta <1$, $0<p_1 < p_2 <\infty$
and $0<r_1, r_2, r \le \infty$. Then the real interpolation
\begin{\eq}   \label{2.31}
\big( L_{p_1, r_1} (\rn), L_{p_2, r_2} (\rn) \big)_{\theta,r} = L_{p,r} (\rn), \qquad \frac{1}{p} = \frac{1-\theta}{p_1} + 
\frac{\theta}{p_2},
\end{\eq}
is covered by \cite[Theorem 1.18.6/2, Remark 1.18.6/5, pp.\,134--135]{T78}. Applied to \eqref{2.17}--\eqref{2.19}, combined with
suitable real interpolations of Besov spaces one obtains
\begin{\eq}   \label{2.32}
F: \quad L_{p,r} (\rn) \hra \Cc^{- \frac{n}{p'}} (\rn), \qquad 1<p<\infty, \quad 0<r \le \infty,
\end{\eq}
\begin{\eq}   \label{2.33}
L_{p,r} (\rn) \os{F}{\hra} L_{p',r} (\rn) \os{\id}{\hra} \Cc^{-\frac{n}{p'}} (\rn) \quad 1<p<2, \quad 0<r \le \infty,
\end{\eq}
and
\begin{\eq}   \label{2.34}
L_{p,r} (\rn) \os{F}{\hra} B^{-n(\frac{1}{2} - \frac{1}{p})}_{2,r} (\rn) \os{\id}{\hra} B^{d^n_p}_{p,r} (\rn) 
\os{\id}{\hra} \Cc^{-\frac{n}{p'}} (\rn), \quad 2<p<\infty, 
\end{\eq}
$0<r \le \infty$.
If one chooses the functions $\psi_j$ in Step 3 of the above proof property then it follows from \eqref{2.30} that \eqref{2.27} remains
valid with $L_{p,r} (\rn)$ in place of $L_p (\rn)$. This shows that the mapping $F$ in \eqref{2.32} is not compact. But one take this
simple observation  also as an illustration of the so--called one--sided compact real interpolation as treated in \cite{CKS92} and
\cite{Cwi92}: If one interpolates the assumed compact mapping \eqref{2.32} with an other suitable continuous mapping of this type such
that one obtains \eqref{2.17} then this mapping must also be compact. But this is not the case as we already know.
\end{remark}

\begin{remark}    \label{R2.7}
As already mentioned in the Introduction, \cite{Ryd20} is the only paper known to us which fits in the above scheme. Let
\begin{\eq}    \label{2.35}
H^s_p (\rn) = F^s_{p,2} (\rn), \qquad 1<p<\infty, \quad s\in \real,
\end{\eq}
be the usual (inhomogeneous) fractional Sobolev spaces. It is one of the main aims of \cite{Ryd20} to justify the continuous mapping
\begin{\eq}    \label{2.36}
F: \quad L_p (\rn) \hra H^{d^n_p}_p (\rn), \qquad 2<p<\infty,
\end{\eq}
with $d^n_p$ as in \eqref{2.16}. From the continuous embedding
\begin{\eq}    \label{2.37}
H^{d^n_p}_p (\rn) \hra B^{d^n_p}_p (\rn), \qquad 2< p <\infty,
\end{\eq}
it follows 
\begin{\eq}    \label{2.38}
F: \quad L_p (\rn) \hra B^{d^n_p}_p (\rn), \qquad 2<p<\infty,
\end{\eq}
what recovers the corresponding assertion in \eqref{2.19}. The one--dimensional case of \eqref{2.38} is applied in \cite{Ryd20} to
study the Laplace transform in the upper complex half--plane related to so--called Laplace--Carleson embeddings and corresponding Carleson measures.
\end{remark}
    
\section{Continuous mappings}    \label{S3}
First  we employ Theorem \ref{T2.5} to deal with the mapping \eqref{1.8} paying special attention to the case when at least one of
these spaces is limiting according to Definition \ref{D2.3}(ii). Let $d^n_p$, $\tau^{n+}_p$ and $\tau^{n-}$ be as in \eqref{2.10},
\eqref{2.11} and \eqref{2.14}, \eqref{2.15}.

\begin{proposition}   \label{P3.1}
Let $1<p<\infty$,
\begin{\eq}   \label{3.1}
X^{s_1}_p (\rn), \ s_1 \ge \tau^{n+}_p, \quad \text{and} \quad Y^{s_2}_p (\rn), \ s_2 \le \tau^{n-}_p,
\end{\eq}
be the spaces as introduced in Definition \ref{D2.3}. Then
\begin{\eq}   \label{3.2}
F: \quad X^{s_1}_p (\rn) \hra Y^{s_2}_p (\rn)
\end{\eq}
is continuous. This mapping is not compact if at least one of the two spaces $X^{s_1}_p (\rn)$ and $Y^{s_2}_p (\rn)$ is limiting.
\end{proposition}

\begin{proof}
{\em Step 1.} It follows from \eqref{2.19} and $FL_2 (\rn) = L_2 (\rn)$ that
\begin{\eq}   \label{3.3}
X^{s_1}_p (\rn) \os{\id}{\hra} L_p (\rn) \os{F}{\hra} B^{d^n_p}_p (\rn) \os{\id}{\hra} Y^{s_2}_p (\rn), \quad 2\le p<\infty.
\end{\eq}
This proves \eqref{3.2} for $2\le p <\infty$. The Fourier transform $F$ is self--dual, $F' =F$, in the dual pairing $\big( S(\rn),
S'(\rn)\big)$ and its restriction to related function spaces. We rely on
\begin{\eq}   \label{3.4}
B^s_p (\rn)' = B^{-s}_{p'} (\rn), \qquad 1<p<\infty, \quad \frac{1}{p} + \frac{1}{p'} =1,
\end{\eq}
\cite[Theorem 2.11.2, p.\,178]{T83}, complemented by $L_p (\rn)' = L_{p'} (\rn)$. With $d^n_{p'} = - d^n_p$ one has
\begin{\eq}   \label{3.5}
X^s_p (\rn)' = Y^{-s}_{p'} (\rn), \ s \ge \tau^{n+}_p \quad \text{and} \quad Y^s_p (\rn)' = X^{-s}_{p'} (\rn), \ s \le \tau^{n-}_p.
\end{\eq}
This covers all spaces of interest as illustrated in Figure \ref{Fig1}. It extends \eqref{3.2} from $2\le p <\infty$ to $1<p<\infty$.
\cm
{\em Step 2.} It remains to prove that $F$ in \eqref{3.2} is not compact if either $X^{s_1}_p (\rn)$ or $Y^{s_2}_p (\rn)$ is limiting.
Let $2\le p <\infty$ and $Y^{s_2}_p (\rn) = B^{d^n_p}_p (\rn)$. By the last embedding in \eqref{2.19}, extended to $p=2$, it follows 
that it is sufficient to show that
\begin{\eq}   \label{3.6}
F: \quad X^{s_1}_p (\rn) \hra \Cc^{-\frac{n}{p'}} (\rn), \qquad s_1 \ge 0,
\end{\eq} 
is not compact. We rely on the same functions $f_j$ as in Step 3 of the proof of Theorem \ref{T2.5} which deals with the case $X^0_p
(\rn) = L_p (\rn)$. Then
\begin{\eq}   \label{3.7}
\| f_j \, | W^m_p (\rn) \| = \sum_{|\alpha| \le m} \| D^\alpha f_j \, | L_p (\rn) \| \sim 1, \qquad j\in \nat,
\end{\eq}
extends \eqref{2.27} to the Sobolev spaces $W^m_p (\rn)$, $m \in \nat$, and, by inclusion,
\begin{\eq}   \label{3.8}
\| f_j \, | B^{s_1}_p (\rn) \| \sim 1, \qquad j\in \nat, \quad s_1 >0.
\end{\eq}
Together with \eqref{2.28}, \eqref{2.29} it follows that $F$ in \eqref{3.6} is not compact. For $2\le p <\infty$ and $X^{s_1}_p (\rn)
= L_p (\rn)$ it is sufficient to show that
\begin{\eq}   \label{3.9}
F: \quad L_p (\rn) \hra \Cc^\sigma (\rn), \qquad \sigma \le - \frac{n}{p'},
\end{\eq}
is not compact. Let $\psi$ be a sufficiently smooth compactly supported starting function in related wavelet expansions as described
in \cite[Section 1.2.2, pp.\,14--17]{T08}. Then $\psi_m (x) = \psi(x-m)$, $m\in \zn$, are also wavelets. Let 
\begin{\eq}   \label{3.10}
f_m (x) = \big( F^{-1} \psi_m \big) (x) = e^{imx} \big( F^{-1} \psi \big) (x), \qquad x\in \rn, \quad m\in \zn.
\end{\eq}
In particular, $\|f_m \, |L_p (\rn) \| = \|\psi^\vee \, | L_p (\rn) \|$ whereas $\big\{ Ff_m = \psi_m: \ m\in \zn \big\}$ is not
compact in $\Cc^\sigma (\rn)$. This covers the desired assertion for $2 \le p <\infty$.
\cm
{\em Step 3.} As far as the $1<p \le 2$ is concerned we rely on the well--known assertion that the linear mapping $T: \ A \hra B$
between the Banach spaces $A$ and $B$ is compact if, and only if, the dual mapping $T': \ B' \hra A'$ is compact, \cite[Theorem 4.19,
p.\,105]{Rud91}. Then it follows from $F= F'$ and \eqref{3.5} that the assertion about non--compactness can be transferred from
$2\le p <\infty$ to $1<p \le 2$.
\end{proof}

The next section deals with the degree of of compactness of the mapping $F$ in \eqref{3.2}, expressed in terms of entropy numbers, if 
neither $X^{s_1}_p (\rn)$ nor $Y^{s_2}_p (\rn)$ is limiting. It is reasonable to fix the qualitative side of these assertions combined
with Proposition \ref{P3.1}.

\begin{theorem}  \label{T3.2}
Let $1<p<\infty$ and let $\tau^{n+}_p$, $\tau^{n-}_p$ be as in \eqref{2.10}. Let
\begin{\eq}   \label{3.11}
X^{s_1}_p (\rn), \ s_1 \ge \tau^{n+}_p, \quad \text{and} \quad Y^{s_2}_p (\rn), \ s_2 \le \tau^{n-}_p,
\end{\eq}
be the spaces as introduced in Definition \ref{D2.3}. Then 
\begin{\eq}   \label{3.12}
F: \quad X^{s_1}_p (\rn) \hra Y^{s_2}_p (\rn)
\end{\eq}
is continuous. This mapping is compact if, and only if, both $s_1 > \tau^{n+}_p$ and $s_2 < \tau^{n-}_p$.
\end{theorem}

\begin{proof}
This follows from Proposition  \ref{P3.1} on the one hand and Theorem \ref{T4.8} below on th other hand.
\end{proof}	

So far we excluded the regions 
\begin{\eq}   \label{3.13}
\big\{ \big( \frac{1}{p},s \big): \ 2<p<\infty, \ d^n_p <s< 0 \big\} \quad \text{and} \quad \big\{ \big(\frac{1}{p}, s \big): \ 1<p<2, 
0<s<d^n_p \big\}
\end{\eq}
with $d^n_p = 2n (\frac{1}{p} - \frac{1}{2} \big)$ as above, both for the source spaces $X^s_p (\rn)$ and the target spaces $Y^s_p (\rn
)$. But it comes out that this is quite natural (at least in the context of unweighted spaces and mappings with fixed $p$). The 
related limiting spaces prove to be barriers.  

Let again $d^n_p$, $\tau^{n+}_p$ and $\tau^{n-}_p$ be as in \eqref{2.10}, \eqref{2.11}.

\begin{corollary}    \label{C3.3}
Let $1<p<\infty$, $s_1 \in \real$ and $s_2 \in \real$. Let
\begin{\eq}   \label{3.14}
F: \quad B^{s_1}_p (\rn) \hra B^{s_2}_p (\rn)
\end{\eq}
be continuous. Then $s_1 \ge \tau^{n+}_p$ and $s_2 \le \tau^{n-}_p$.
\end{corollary}

\begin{proof}
Let $2\le p <\infty$ and $s_1 <0$. Then one may assume in \eqref{3.14} that $s_2 <d^n_p$. Combined with 
\begin{\eq}   \label{3.15}
F: \quad B^{s_0}_p (\rn) \hra B^{s_2}_p (\rn), \qquad s_0>0, \quad \text{compact},
\end{\eq}
one obtains by the classical real interpolation, \cite[Theorem 1.16.4/1, p.\,117]{T78}, and embeddding that also
\begin{\eq}   \label{3.16}
F: \quad L_p (\rn) \hra B^{s_2}_p (\rn) \qquad \text{is compact.}
\end{\eq}
But this contradicts Proposition \ref{P3.1}. Let $2 \le p <\infty$ and $s_2 > d^n_p$. Then we may assume in \eqref{3.14} that $s_1   >0$. Combined with
\begin{\eq}   \label{3.17}
F: \quad B^{s_1}_p (\rn) \hra B^{s_0}_p (\rn), \qquad s_0 <d^n_p, \quad \text{compact},
\end{\eq} 
one obtains by real interpolation that also
\begin{\eq}   \label{3.18}
F: \quad B^{s_1}_p (\rn) \hra B^{d^n_p}_p (\rn) \qquad \text{is compact}.
\end{\eq}
But this contradicts again Proposition \ref{P3.1}. One can argue similarly if $1<p<2$ (or rely on duality).
\end{proof}

\begin{remark}   \label{R3.4}
Theorem  \ref{T3.2} and Corollary \ref{C3.3} show that the regions \eqref{2.14} and \eqref{2.15} are very natural for mappings 
generated by $F$ between (unweighted) function spaces of the above type with fixed $p$.
\end{remark}

\section{Compact mappings}    \label{S4}
\subsection{Entropy numbers and weighted spaces}    \label{S4.1}
Proposition  \ref{P3.1} and Corollary \ref{C3.3} show that \eqref{2.14} and \eqref{2.15} are natural regions in the $\big( \frac{1}{p},
s\big)$-half-plane for mappings $F$ generated by the Fourier transform. As already indicated in Theorem \ref{T3.2} the mapping $F$ in
\eqref{3.12} is even compact if, and only if, $s_1 > \tau^{n+}_p$ and $s_2 < \tau^{n-}_p$. We wish to justify this claim and to
discuss the degree of compactness in terms of entropy numbers. This will be reduced to compact embeddings between related weighted
spaces. We collect what we need following \cite[Chapter 6]{T06} which in turn is based on \cite[Chapter 4]{ET96} and \cite{HaT94a},
\cite{HaT94b}, \cite{HaT05}. But first we recall what is meant by entropy numbers.

\begin{definition}   \label{D4.1}
Let $T: \ A \hra B$ be a linear and continuous mapping from the quasi--Banach space $A$ into the the quasi--Banach space $B$. Then
the entropy number $e_k (T)$, $k\in \nat$, is the infimum of all $\ve >0$ such that
\begin{\eq}  \label{4.1}
T (U_A) \subset \bigcup^{2^{k-1}}_{j=1} ( b_j + \ve U_B) \quad \text{for some $b_1, \ldots, b_{2^{k-1}} \in B$,}
\end{\eq}
where $U_A = \{ a\in A: \ \|a \, | A \| \le 1 \}$ and $U_B = \{ b\in B: \ \|b \, | B \| \le 1 \}$.
\end{definition}

\begin{remark}   \label{R4.2}
Basic properties and related references may be found in \cite[Section 1.10, pp.\,55--58]{T06}. We only mention that the linear and
continuous mapping $T: \, A \hra B$ is compact if, and only if, $e_k (T) \to 0$ for $k \to \infty$.
\end{remark}

We need the weighted counterpart of $L_p (\rn)$, quasi--normed according to \eqref{2.1}: The quasi--Banach space $L_p (\rn, w_\alpha)$,
$0<p \le \infty$ with
\begin{\eq}  \label{4.2}
w_\alpha (x) = (1 + |x|^2)^{\alpha/2}, \qquad \alpha \in \real, \quad x\in \rn,
\end{\eq}
consists of all (complex--valued) Lebesgue--measurable functions in $\rn$ such that
\begin{\eq}   \label{4.3}
\|f \, | L_p (\rn, w_\alpha) \| = \| w_\alpha f \, | L_p (\rn) \|
\end{\eq}
is finite. Then Definition \ref{D2.1} can be extended as follows.

\begin{definition}   \label{D4.3}
Let $\vp = \{\vp_j \}^\infty_{j=0}$ be the dyadic resolution of unity as introduced in \eqref{2.3}--\eqref{2.5}. Let
\begin{\eq}   \label{4.4}
s\in \real, \quad \alpha \in \real \quad \text{and} \quad 0<p,q \le \infty.
\end{\eq}
Then $\Bs (\rn, w_\alpha)$ is the collection of all $f\in S'(\rn)$ such that
\begin{\eq}   \label{4.5}
\|f \, | \Bs (\rn, w_\alpha) \|_{\vp} 
= \Big( \sum^\infty_{j=0} 2^{jsq} \big\| (\vp_j \wh{f})^\vee |L_p (\rn, w_\alpha) \big\|^q \Big)^{1/q}
\end{\eq}
is finite $($with the usual modification if $q=\infty)$.
\end{definition}

\begin{remark}   \label{R4.4}
For details and properties one may consult the above references. These quasi--Banach spaces are independent of $\vp$ (equivalent 
quasi--norms) what justifies the omission of the subscript $\vp$ in the sequel. We only mention that $f \to w_\alpha f$ is an 
isomorphic mapping of $\Bs (\rn, w_\alpha)$ onto $\Bs (\rn)$ and
\begin{\eq}   \label{4.6}
\| w_\alpha f \, | \Bs (\rn) \| \sim \| f \, | \Bs (\rn, w_\alpha) \|, \qquad \alpha \in \real,
\end{\eq}
(equivalent quasi--norms) in good agreement with \eqref{4.3}.
\end{remark}

Next we recall what is known about the entropy numbers $e_k (\id)$, $k\in \nat$, of the compact embeddings between the spaces 
introduced above,
\begin{\eq}   \label{4.7}
\id: \quad B^{s_1}_{p_1, q_1} (\rn, w_\alpha) \hra B^{s_2}_{p_2, q_2} (\rn).
\end{\eq}
Let $s_1 \in \real$, $s_2 \in \real$, $\alpha \ge 0$ and $p_1, p_2, q_1, q_2 \in (0, \infty]$. Let
\begin{\eq}   \label{4.8}
\delta = s_1 - \frac{n}{p_1} - \big( s_2 - \frac{n}{p_2} \big) \quad \text{and} \quad \frac{1}{p_*} = \frac{1}{p_1} + 
\frac{\alpha}{n}.
\end{\eq}
Then $\id$ in \eqref{4.7} is compact if, and only if,
\begin{\eq}   \label{4.9}
s_1 > s_2, \quad \delta >0, \quad \alpha >0, \quad p_2 > p_*.
\end{\eq}
This coincides with \cite[Proposition 6.29, p.\,281]{T06}. According to \cite[Theorem 6.31, pp.\,282--283, Theorem 6.33, p.\,284]{T06}
one has the following (almost) final assertions for the related entropy numbers.

\begin{figure}[t]
\begin{minipage}{\textwidth}
~\hfill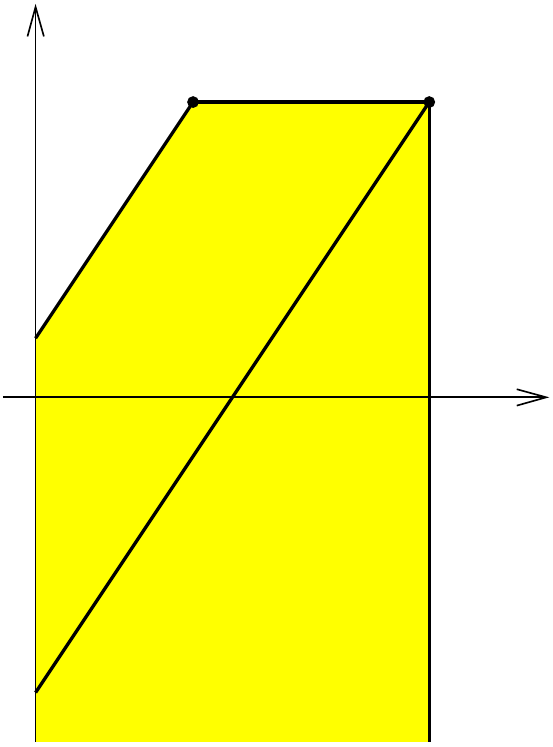\hfill~\\
\end{minipage}
Figure 2: Entropy numbers, general case   \label{Fig2}
\end{figure}

\begin{proposition}   \label{P4.5}
Let $e_k (\id)$, $k\in \nat$, be the entropy numbers of the compact embeddings \eqref{4.7}--\eqref{4.9}. Then, as illustrated in 
Figure $2$, 
\begin{\eq}   \label{4.10}
e_k (\id) \sim k^{-\frac{s_1 - s_2}{n}}, \quad k \in \nat, \qquad \text{if} \quad \delta < \alpha,
\end{\eq}
and
\begin{\eq}   \label{4.11}
e_k (\id) \sim k^{- \frac{\alpha}{n} + \frac{1}{p_2} - \frac{1}{p_1}}, \quad k\in \nat, \qquad \text{if} \quad \delta > \alpha.
\end{\eq}
Let, in addition,
\begin{\eq}   \label{4.12}
\vr = \frac{s_1 - s_2}{n} + \frac{1}{q_2} - \frac{1}{q_1}.
\end{\eq}
Then
\begin{\eq}   \label{4.13}
e_k (\id) \sim k^{- \frac{s_1 - s_2}{n}}, \quad k\in \nat, \qquad \text{if} \quad \delta = \alpha, \quad \vr <0,
\end{\eq}
and
\begin{\eq}   \label{4.14}
e_k (\id) \sim k^{-\frac{s_1 - s_2}{n}} \big( \log k)^{\vr}, \quad 1<k\in \nat, \qquad \text{if} \quad \delta = \alpha, \quad \vr>0.
\end{\eq}
\end{proposition}

\begin{remark}   \label{R4.6}
This is sufficient for our later purposes. Further information and discussions may be found in \cite[Section 6.4, pp.\,279--286]{T06}.
\end{remark}

We are mainly interested in the source spaces
\begin{\eq}   \label{4.15}
L_2 (\rn, w_\alpha) = B^0_{2,2} (\rn, w_{\alpha}), \qquad \alpha >0,
\end{\eq}
normed by
\begin{\eq}   \label{4.16}
\| f \, | L_2 (\rn, w_\alpha) \| = \| w_\alpha f \, | L_2 (\rn) \|
\end{\eq}
as in \eqref{4.3} and the target spaces
\begin{\eq}    \label{4.17}
B^s_p (\rn) = B^s_{p,p} (\rn), \quad s<0, \quad \delta = n \big(\frac{1}{p} - \frac{1}{2} \big) -s >0, \quad \frac{1}{p} < 
\frac{1}{p_*} = \frac{1}{2} + \frac{\alpha}{n},
\end{\eq}
with
\begin{\eq}   \label{4.18}
H^s (\rn) = B^s_2 (\rn), \qquad s<0,
\end{\eq}
as special case. Recall the well-known
classical assertion that $f\in S' (\rn)$ belongs to $H^\sigma (\rn)$, $\sigma \in \real$, if, and only if,
$(w_\sigma \wh{f} )^\vee \in L_2 (\rn)$ with $w_\sigma$ as in \eqref{4.2} and
\begin{\eq}   \label{4.19}
\| f \, | H^\sigma (\rn) \| \sim \| (w_\sigma \wh{f} )^\vee | L_2 (\rn) \| = \| w_\sigma \wh{f} \, | L_2 (\rn) \|.
\end{\eq}

\begin{figure}[t]
\begin{minipage}{\textwidth}
~\hfill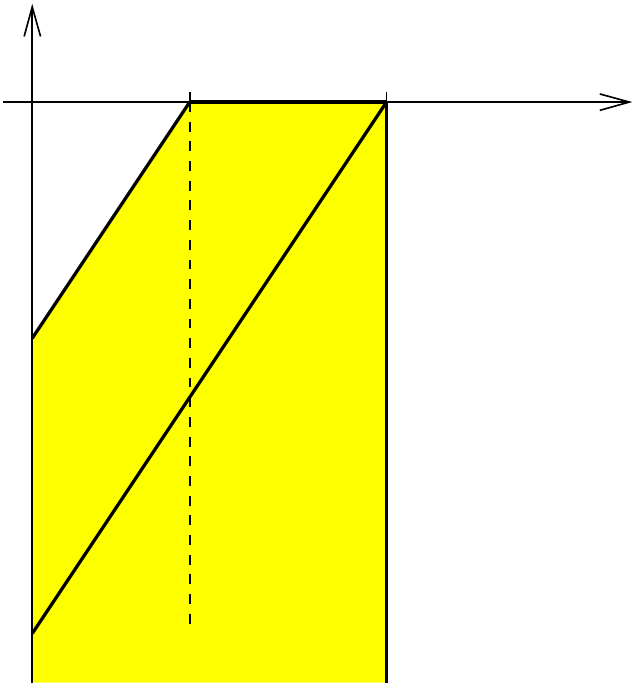\hfill~\\
\end{minipage}
Figure 3: Entropy numbers, special case   \label{Fig3}
\end{figure}

\begin{corollary}   \label{C4.7}
Let $\alpha >0$ and $s<0$.
\cm
{\em (i)} Let, in addition,
\begin{\eq}    \label{4.20}
0 \le \frac{1}{p} < \frac{1}{2} + \frac{\alpha}{n} \quad \text{and} \quad \delta = n \big(\frac{1}{p} - \frac{1}{2} \big) -s >0.
\end{\eq}
Then, as illustrated in Figure  $3$, 
\begin{\eq}   \label{4.21}
\id: \quad L_2 (\rn, w_\alpha) \hra B^s_p (\rn)
\end{\eq}
is compact and
\begin{\eq}    \label{4.22}
e_k (\id) \sim
\begin{cases}
k^{\frac{s}{n}} &\text{if $\delta < \alpha$}, \\
k^{\frac{s}{n}} (\log k)^{\frac{\alpha}{n}} &\text{if $\delta = \alpha$}, \\
k^{- \frac{\alpha}{n} + \frac{1}{p} - \frac{1}{2}} &\text{if $\delta> \alpha$},
\end{cases}
\end{\eq}
$2 \le k \in \nat$.
\cm
{\em (ii)} Then
\begin{\eq}    \label{4.23}
\id: \quad L_2 (\rn, w_\alpha) \hra H^s (\rn)
\end{\eq}
is compact and
\begin{\eq}   \label{4.24}
e_k (\id) \sim
\begin{cases}
k^{\frac{s}{n}} &\text{if $-s < \alpha$}, \\
\big( \frac{k}{\log k} \big)^{\frac{s}{n}} &\text{if $-s = \alpha$}, \\
k^{- \frac{\alpha}{n}} &\text{if $-s> \alpha$},
\end{cases}
\end{\eq}
$2\le k \in \nat$.
\end{corollary}

\begin{proof}
Part (i) follows from Proposition \ref{P4.5} and part (ii) follows from part (i).
\end{proof}

\subsection{Main assertions}   \label{S4.2}
Compared with Theorem \ref{T3.2} it remains to prove that the mapping $F$ in \eqref{3.12} is compact if $s_1 > \tau^{n+}_p$ and $s_2
< \tau^{n-}_p$, where again
\begin{\eq}  \label{4.25}
\tau^{n+}_p = \max (0, d^n_p), \quad \tau^{n-}_p = \min (0, d^n_p ) \quad \text{and} \quad d^n_p = 2n \big( \frac{1}{p} - \frac{1}{2}
\big).
\end{\eq}
This applies by Definition \ref{D2.3} to
\begin{\eq}   \label{4.26}
X^{s_1}_p (\rn) = B^{s_1}_p (\rn), \qquad 1<p<\infty, \quad s_1 > \tau^{n+}_p,
\end{\eq}
\begin{\eq}   \label{4.27}
Y^{s_2}_p (\rn) = B^{s_2}_p (\rn), \qquad 1<p<\infty, \quad s_2 < \tau^{n-}_p
\end{\eq}
and, with $H^\sigma (\rn) = B^\sigma_2 (\rn) = B^\sigma_{2,2} (\rn)$, $\sigma \in \real$, to
\begin{\eq}   \label{4.28}
X^{s_1}_2 (\rn) = H^{s_1} (\rn), \ s_1 >0 \quad \text{and} \quad Y^{s_2}_2 (\rn) = H^{s_2} (\rn), \ s_2 <0.
\end{\eq}
The degree of compactness will be measured in terms of entropy numbers as introduced in Definition \ref{D4.1}. Recall that $a_k \sim
b_k$ (equivalence) for two sets $\{ a_k: \, k\in \nat \}$ and $\{ b_k: \, k\in \nat \}$ of positive numbers means $0 < \inf_{k\in \nat}
a_k b^{-1}_k  \le \sup_{k\in \nat} a_k b^{-1}_k <\infty$. 

\begin{theorem}   \label{T4.8}
Let $1<p<\infty$ and let $\tau^{n+}_p$, $\tau^{n-}_p$, $d^n_p$ be as in \eqref{4.25}. Let $X^{s_1}_p (\rn)$ and $Y^{s_2}_p (\rn)$ be
the spaces as introduced in Definition \ref{D2.3}.
\cm
{\em (i)} Let $s_1 > \tau^{n+}_p$ and $s_2 < \tau^{n-}_p$. Then
\begin{\eq}    \label{4.29}
F: \quad X^{s_1}_p (\rn) \hra Y^{s_2}_p (\rn)
\end{\eq}
is compact.
\cm
{\em (ii)} Let $p=2$ and $s_1 >0>s_2$. Then
\begin{\eq}    \label{4.30}
F: \quad H^{s_1} (\rn) \hra H^{s_2} (\rn)
\end{\eq} 
is compact and
\begin{\eq}   \label{4.31}
e_k (F) \sim
\begin{cases}
k^{\frac{s_2}{n}} &\text{if $s_2 > -s_1$}, \\
\big( \frac{k}{\log k} \big)^{\frac{s_2}{n}} &\text{if $s_2 = -s_1$}, \\
k^{- \frac{s_1}{n}} &\text{if $s_2 < -s_1$},
\end{cases}
\end{\eq}
$2\le k \in \nat$.
\cm
{\em (iii)} Let $1<p<2$, $s_1 > d^n_p$ and $s_2 <0$ as indicated in Figure $1$. Then
\begin{\eq}   \label{4.32}
F: \quad B^{s_1}_p (\rn) \hra  B^{s_2}_p (\rn)
\end{\eq}
is compact and
\begin{\eq}    \label{4.33}
e_k (F) \le c
\begin{cases}
k^{\frac{s_2}{n}} &\text{if $s_2 > d^n_p -s_1$}, \\
\big( \frac{k}{\log k} \big)^{\frac{s_2}{n}} (\log k)^{\frac{1}{p} - \frac{1}{2}} &\text{if $s_2 = d^n_p -s_1$}, \cm
k^{- \frac{s_1}{n} + 2(\frac{1}{p} - \frac{1}{2})} &\text{if $s_2 < d^n_p -s_1$},
\end{cases}
\end{\eq}
for some $c>0$ and $2 \le k \in \nat$.
\cm
{\em (iv)} Let $2<p<\infty$, $s_1 >0$ and $s_2 <d^n_p$ as indicated in Figure $1$. Then
\begin{\eq}  \label{4.34}
F: \quad B^{s_1}_p (\rn) \hra B^{s_2}_p (\rn)
\end{\eq}
is compact and
\begin{\eq}    \label{4.35}
e_k (F) \le c
\begin{cases}
k^{\frac{s_2}{n} - 2(\frac{1}{p} - \frac{1}{2})} &\text{if $s_2 > d^n_p -s_1 $}, \\
\big( \frac{k}{\log k} \big)^{-\frac{s_1}{n}} (\log k)^{\frac{1}{2} - \frac{1}{p}} &\text{if $s_2 = d^n_p -s_1 $}, \cm
k^{- \frac{s_1}{n}} &\text{if $s_2 < d^n_p -s_1$},
\end{cases}
\end{\eq}
for some $c>0$ and $2 \le k \in \nat$.
\end{theorem}

\begin{proof}
{\em Step 1.} Part (i) is a by--product of the parts (ii)--(iv) and \eqref{4.26}--\eqref{4.28}.
\cm
{\em Step 2.} We prove part (ii). Let $L_2 (\rn, w_\alpha)$ be the weighted Hilbert space according to \eqref{4.2}, \eqref{4.3}. Let
$f\in H^{s_1} (\rn)$. Then one has by \eqref{4.19} and
\begin{\eq}   \label{4.36}
\| Ff \, | L_2 (\rn, w_{s_1}) \|  = \| w_{s_1} \wh{f} \, | L_2 (\rn) \|   = \| f \, | H^{s_1} (\rn) \|
\end{\eq}
that $Ff \in L_2 (\rn, w_{s_1})$.
Conversely if $g\in L_2 (\rn, w_{s_1})$ and $Ff =g$, then it follows from \eqref{4.36} that $f\in H^{s_1} (\rn)$. This shows that $F$
is an isomorphic mapping of $H^{s_1} (\rn)$ onto $L_2 (\rn, w_{s_1})$,
\begin{\eq}   \label{4.37}
F H^{s_1} (\rn) = L_2 (\rn, w_{s_1}).
\end{\eq}
Combined with Corollary \ref{C4.7}(ii) one obtains \eqref{4.31} for the mapping \eqref{4.30}.
\cm
{\em Step 3.} We prove part (iii). We rely on the limiting embedding
\begin{\eq}   \label{4.38}
\id: \quad B^{s_1}_p (\rn) \hra B^\alpha_{2,p} (\rn) \hra H^\alpha (\rn), \qquad 1<p<2,
\end{\eq}
with
\begin{\eq}   \label{4.39}
\alpha = s_1 - n \big( \frac{1}{p} - \frac{1}{2} \big) > n \big( \frac{1}{p} - \frac{1}{2} \big) >0.
\end{\eq}
Then one obtains from \eqref{4.37} (with $\alpha$ in place $s_1$) that
\begin{\eq}   \label{4.40}
F: \quad B^{s_1}_p (\rn) \hra L_2 (\rn, w_\alpha).
\end{\eq}
Application of part (i) of Corollary \ref{C4.7} with $s_2$ in place of $s$ and with $\alpha$ as \eqref{4.39} prove \eqref{4.33}.
\cm
{\em Step 4.} The proof of part (iv) relies on duality. Let $T: A \hra B$ be a linear and compact mapping from the complex Banach space
$A$ into the complex Banach space $B$ and let $e_k (T)$, $k\in \nat$, be the corresponding entropy numbers. Let $T': B' \hra A'$ be
the dual operator. Then there are numbers $l\in \nat$ and $c \ge 1$ such that
\begin{\eq}   \label{4.41}
e_{lk} (T') \le c\, e_k (T), \qquad k\in \nat,
\end{\eq}
if at least one of the two spaces $A$ and $B$ is isomorphic to $\ell_p$  with $1<p<\infty$. This remarkable assertion goes back to
\cite{AMS04} and \cite{AMST04} reformulated in \cite[pp.\,332--333]{Pie07} in the above way. It has already been observed in  
\cite[Theorem 2.11.2(b), p.\,237]{T78} that $B^s_p (\rn) = B^s_{p,p} (\rn)$, $1<p<\infty$, is isomorphic to $\ell_p$. It came out later
on  that this property remains valid for all $p$ with $0<p \le \infty$, \cite[Corollary 3.8, p.\,157]{T06}. It is also an immediate 
consequence of related wavelet isomorphisms, \cite[Theorem 1.20, pp.\,15--16]{T08}. Now one can rely part (iii) and the duality
\begin{\eq}   \label{4.42}
B^s_p (\rn)' = B^{-s}_{p'} (\rn), \qquad s\in \real, \quad 1 \le p<\infty, \quad \frac{1}{p} + \frac{1}{p'} =1,
\end{\eq}
within the dual pairing $\big( S(\rn), S'(\rn) \big)$
\cite[Theorem 2.11.2, p.\,178]{T83}. Let $1<p<2$ and let temporarily $s_1 > d^n_p$, $s_2 <0$ as in \eqref{4.32}, \eqref{4.33}. Let
$\sigma_1 = -s_2$ and $\sigma_2 = - s_1$. Then
\begin{\eq}   \label{4.43}
\sigma_1 >0 \quad \text{and} \quad \sigma_2 < -d^n_p = d^n_{p'}.
\end{\eq}
The Fourier transform $F$ is self--dual, $F' =F$, in the dual pairing $\big( S(\rn), S'(\rn) \big)$ and in its restrictions to the above
Banach spaces. Then it follows from the above duality for entropy numbers that for some $c>0$ and all $k\in \nat$,
\begin{\eq}   \label{4.44}
e_k \big(F: \ B^{\sigma_1}_{p'} (\rn) \hra B^{\sigma_2}_{p'} (\rn) \big) \le c \, e_k \big(F: B^{s_1}_p (\rn) \hra B^{s_2}_p (\rn)
\big).
\end{\eq}
Using \eqref{4.33} one obtains \eqref{4.35} with $\sigma_1$, $\sigma_2$, $p'$ in place of $s_1$, $s_2$, $p$. Reformulation results in
the desired assertion.
\end{proof}

\begin{remark}    \label{R4.9}
Step 3 of the above proof shows that part (iii) of the above theorem remains valid if one extends $1<p<2$ to $0<p<2$. But we preferred
$1<p<2$ because this restriction is consistent with the above set--up in \eqref{4.26}, \eqref{4.27} and \eqref{4.29}.
\end{remark}

We are mainly interested in the interplay of already available  ingredients. This may justify to touch briefly on interpolation. We assume that the reader is familiar with basic assertions of interpolation theory. Let $A$ be a quasi--Banach space  and let $\{B_1, 
B_2 \}$ be an interpolation couple  of quasi--Banach spaces. Let both $T: A \hra B_1$ and $T: A \hra B_2$ be linear and compact. Then
\begin{\eq}   \label{4.45}
T: \quad A \hra B_{\theta,q} = (B_1, B_2 )_{\theta, q}, \qquad 0< \theta <1, \quad 0<q \le \infty,
\end{\eq}
is compact and one has
\begin{\eq}   \label{4.46}
e_{k_1 + k_2 -1} (T: A\hra B_{\theta,q} ) \le c\, e_{k_1} (T: A \hra B_1 )^{1-\theta} e_{k_2} (T: A \hra B_2 )^\theta,
\end{\eq}
for some $c>0$ and all $k_1 \in \nat$, $k_2 \in \nat$ for the related entropy numbers. Let $\{A_1, A_2 \}$ be an interpolation couple
of quasi--Banach spaces and let $B$ be a quasi--Banach space. Let both $T: A_1 \hra B$ and $T: A_2 \hra B$ be linear and compact. Then
\begin{\eq}   \label{4.47}
T: \quad (A_1, A_2)_{\theta,q} = A_{\theta,q} \hra B, \qquad 0<\theta <1, \quad 0<q \le \infty,
\end{\eq}
is compact and one has
\begin{\eq}   \label{4.48}
e_{k_1 + k_2 -1} (T: A_{\theta,q}\hra B) \le c\, e_{k_1} (T: A_1 \hra B)^{1-\theta} e_{k_2} (T: A_2 \hra B)^\theta,
\end{\eq}
for some $c>0$ and all $k_1 \in \nat$, $k_2 \in \nat$ for the related entropy numbers. The short proof of these assertions in 
\cite[Section 1.3.2, pp.\,13--15]{ET96} goes back to \cite{HaT94a}. An extension of this interpolation property to $T: A_{\theta,q}
\hra B_{\theta,q}$ for related interpolation couples $\{A_1, A_2 \}$ and $\{B_1, B_2 \}$ is not possible in general. This problem was 
open for a long time and had been finally settled in \cite{EdN11} and \cite{EdN13}. This shows that that the qualitative so--called
one--sided compact interpolation mentioned briefly in Remark \ref{R2.6} has no quantitative counterpart in terms of entropy numbers.  But the assertions \eqref{4.46} and \eqref{4.48} are sufficient  for our purpose, based on the following well-known real
interpolation. Let $-\infty <s_1 <s_2 <\infty$ and $0<p, q_1, q_2, q \le \infty$. Let $0<\theta <1$ and $s= (1-\theta)s_1 + \theta s_2.
$ Then
\begin{\eq}   \label{4.49}
\big( B^{s_1}_{p, q_1} (\rn), B^{s_2}_{p, q_2} (\rn) \big)_{\theta,q} = \Bs (\rn),
\end{\eq}
\cite[Theorem 2.4.2, p.\,64]{T83} extending \cite[Theorem 2.4.1, p.\,182]{T78} from $1<p<\infty$ and $1\le q_1, q_2, q \le \infty$ to
the above parameters. Let again $d^n_p = 2n \big( \frac{1}{p} - \frac{1}{2} \big)$ for $n\in \nat$ and $0<p \le \infty$.

\begin{corollary}   \label{C4.10}
{\em (i)} Let $1<p \le 2$, $s_1 > d^n_p$ and $s_2 <0$. Let $0 <q_1, q_2 \le \infty$. Then
\begin{\eq}  \label{4.50}
F: \quad B^{s_1}_{p, q_1} (\rn) \hra B^{s_2}_{p, q_2} (\rn)
\end{\eq}
is compact and
\begin{\eq}    \label{4.51}
e_k (F) \le c
\begin{cases}
k^{\frac{s_2}{n}} &\text{if $s_2 > d^n_p-s_1$}, \\
k^{- \frac{s_1}{n} + 2(\frac{1}{p} - \frac{1}{2})} &\text{if $s_2 < d^n_p-s_1$},
\end{cases}
\end{\eq}
for some $c>0$ and all $k\in \nat$.
\cm
{\em (ii)} Let $2\le p<\infty$, $s_1 >0$, $s_2 <d^n_p$. Let $0<q_1, q_2 \le \infty$. Then
\begin{\eq}  \label{4.52}
F: \quad B^{s_1}_{p, q_1} (\rn) \hra B^{s_2}_{p,q_2} (\rn)
\end{\eq}
is compact and
\begin{\eq}    \label{4.53}
e_k (F) \le c
\begin{cases}
k^{\frac{s_2}{n} - 2(\frac{1}{p} - \frac{1}{2})} &\text{if $s_2 > d^n_p-s_1$}, \\
k^{- \frac{s_1}{n}} &\text{if $s_2 < d^n_p-s_1$},
\end{cases}
\end{\eq}
for some $c>0$ and all $k \in \nat$.
\end{corollary}

\begin{proof} Let, say, $1<p\le 2$ and $A = B^{s_1}_p (\rn)$ with $s_1 > d^n_p$. Choosing for $B_1$ and $B_2$ two different spaces
$B^{s_2}_{p_2} (\rn)$ with, say, $0>s_2 > d^n_p -s_1$, then it follows from \eqref{4.33} and the interpolation \eqref{4.45}, 
\eqref{4.46}, based on \eqref{4.49}, that
\begin{\eq}    \label{4.54}
e_k \big( F: \ B^{s_1}_p (\rn) \hra B^{s_2}_{p, q_2} (\rn) \big) \le c\, k^{\frac{s_2}{n}}, \qquad k \in \nat,
\end{\eq}
$0<q_2 \le \infty$. For fixed $B^{s_2}_{p, q_2} (\rn)$ one obtains by the same type of argument, now based on \eqref{4.48}, that
\begin{\eq}   \label{4.55}
e_k \big(F: \ B^{s_1}_{p,q_1} (\rn) \hra B^{s_2}_{p, q_2} (\rn) \big) \le c \, k^{\frac{s_2}{n}}, \qquad k \in \nat.
\end{\eq}
The other cases can be treated similarly.
\end{proof}

\begin{remark}   \label{R4.11}
The well--known embedding
\begin{\eq}   \label{4.56}
B^s_{p, \min(p,q)} (\rn) \hra \Fs (\rn) \hra B^s_{p, \max(p,q)} (\rn),
\end{\eq}
$0<p<\infty$, $0<q \le \infty$, $s\in \real$, shows that one can replace $B^{s_r}_{p, q_r} (\rn)$ in the above corollary by 
$F^{s_r}_{p, q_r} (\rn)$ with $r=1$ and $r=2$. This applies in particular to the fractional Sobolev spaces
\begin{\eq}   \label{4.57}
H^s_p (\rn) = F^s_{p,2} (\rn), \qquad 1<p<\infty, \quad s\in \real.
\end{\eq}
The resulting  estimate for the entropy numbers of
\begin{\eq}   \label{4.58}
F: \quad H^{s_1}_p (\rn) \hra H^{s_2}_p (\rn)
\end{\eq}
in \eqref{4.51} and \eqref{4.53} may be considered as a complement to \eqref{2.36} going back to \cite{Ryd20}. If $p=2$ then one has
the equivalence \eqref{4.31}.
\end{remark}

\section{Some problems}    \label{S5}
We list some problems which might be of interest for further  research.

\begin{problem}   \label{P5.1}
If $p=2$ then one has for the mapping \eqref{4.29} the equivalence \eqref{4.31}. The problem arises to show that \eqref{4.33} and
\eqref{4.35} are also equivalences.
\end{problem}

\begin{problem}   \label{P5.2}
The choice $X^0_p (\rn) = L_p (\rn)$ if $2 \le p <\infty$ and $Y^0_p (\rn) = L_p (\rn)$ if $1<p \le 2$ in Definition \ref{D2.3}
is natural. Theorem \ref{T3.2}
and Corollary \ref{C3.3} justify that we called these spaces in Definition \ref{D2.3}(ii) limiting. The question arises whether one can
replace $L_p (\rn)$, $1<p<\infty$, in Theorem \ref{T3.2} by $B^0_p (\rn)$ or $B^0_{p,q} (\rn)$ including related complements in
Corollary \ref{C3.3}.
\end{problem}

\begin{problem}   \label{P5.3}
In contrast to Theorem \ref{T4.8} we had to exclude in Corollary \ref{C4.10}  the case $s_1 + s_2 =d^n_p$. One may ask what happens
in this limiting situation, expecting that $q_1$ and $q_2$ come in (taking \eqref{4.13}, \eqref{4.14} as a guide).
\end{problem}

\begin{problem}   \label{P5.4}
The study  of mapping properties of the Fourier transform in unweighted spaces $B^s_p (\rn)$, $\Bs (\rn)$, $H^s_p (\rn)$ and also
$\Fs (\rn)$ is quite natural. It is the first choice (beyond the classical assertions \eqref{1.2} and \eqref{1.3}). But our approach
uses already the weighted spaces $\Bs (\rn, w_\alpha )$ as introduced in Definition \ref{D4.3}. In addition one has for fixed $0<p,q \le \infty$ that
\begin{\eq}   \label{5.1}
S(\rn) = \bigcap_{\alpha \in \real, s\in \real} \Bs (\rn, w_\alpha) \quad \text{and} \quad 
S'(\rn) = \bigcup_{\alpha \in \real, s\in \real} \Bs (\rn, w_\alpha).
\end{\eq}
This is more or less known
and may also be found in \cite[(2.281), p.\,74]{T20} with a reference to \cite{Kab08} for a detailed proof. These observations  
suggest strongly to use these weighted spaces  not only as a tool (in connection with entropy numbers) but as natural spaces to study
mapping properties of Fourier transforms between them.
\end{problem}

\begin{problem}   \label{P5.5}
Instead of entropy numbers to measure the compactness  of mappings of the Fourier transform between suitable function spaces one may
rely on other distinguished numbers, above all approximation numbers. The underlying assertions in terms of the spaces $\Bs (\rn,
w_\alpha)$ have been mentioned in \cite[Section 4.3.3, pp.\,179--184]{ET96} with a reference to \cite{Har95}. This has been 
complemented in \cite{Cae98} and \cite{Skr05}.
\end{problem}

\begin{problem}   \label{P5.6}
We discussed limiting situations only if they illuminate what is going on, not for their own sake. Nevertheless some of them might be of interest. Whereas the limiting case $p=1$ is largely covered by Remark \ref{R4.9} the situation is different if $p= \infty$. It follows from the duality
\eqref{4.42} and Theorem \ref{T4.8}(iii) extended to $p=1$ that
\begin{\eq}  \label{5.2}
F: \quad \Cc^{s_1} (\rn) \hra \Cc^{s_2} (\rn) \qquad \text{with} \quad s_2 +n <0 < s_1
\end{\eq}
is compact, but it is not clear whether the estimate \eqref{4.35} for the related entropy numbers  can be extended to $p=\infty$ with the
excepted outcome 
\begin{\eq}    \label{5.3}
e_k (F) \le c
\begin{cases}
k^{\frac{s_2}{n} +1} &\text{if $s_1+s_2+n >0$}, \\
\big( \frac{k}{\log k} \big)^{-\frac{s_1}{n}} \sqrt{\log k} &\text{if $s_1 +s_2 +n =0 $}, \cm
k^{- \frac{s_1}{n}} &\text{if $s_1 +s_2 +n <0$},
\end{cases}
\end{\eq}
for some $c>0$ and $2 \le k \in \nat$, or even equivalence. 
\end{problem}
{\bfseries Acknowledgement.} I would like to thank Dorothee D. Haroske for producing  the figures.


\begin{thebibliography}{iiiiii}

\bibitem[AMS04]{AMS04} S. Artstein, V. Milman, S.J. Szarek. Duality of metric entropy. Ann. of Math. {\bfseries 159} (2004), 
1313--1328.

\bibitem[AMST04]{AMST04} S. Artstein, V. Milman, S. Szarek, N. Tomczak--Jaegermann. On convexified  packing and entropy duality.
GAFA, Geom. Funct. Anal. {\bfseries 14} (2004), 1134--1141.

\bibitem[Cae98]{Cae98} About approximation numbers in function spaces. J. Approx. Theory {\bfseries 94} (1998), 383--395.

\bibitem[CKS92]{CKS92} F. Cobos, T. K\"{u}hn, T. Schonbek. One--sided compactness results for Aronszajn--Gagliardo functors. J. Funct.
Anal. {\bfseries 106} (1992), 274--313.

\bibitem[Cwi92]{Cwi92} M. Cwikel. Real and complex interpolation and extrapolation of compact operators. Duke Math. J. {\bfseries 65}
(1992), 333--343.

\bibitem[DoV21]{DoV21} O. Dominguez, M. Veraar. Extensions of vector--valued Hausdorff--Young inequalities. Math. Z. {\bfseries 299}
(2021), 373--425.

\bibitem[EdN11]{EdN11} D.E. Edmunds, Yu. Netrusov. Entropy numbers and interpolation. Math. Ann. {\bfseries 351} (2011), 963--977.

\bibitem[EdN13]{EdN13} D.E. Edmunds, Yu. Netrusov. Entropy numbers of operators acting between vector--valued sequence spaces. Math.
Nachr. {\bfseries 286} (2013), 614--630.

\bibitem[ET96]{ET96} D.E. Edmunds, H. Triebel. Function spaces, entropy numbers, differential operators. Cambridge Univ. Press,
Cambridge, 1996.

\bibitem[Har95]{Har95} D.D. Haroske. Approximation numbers in some weighted function spaces. J. Approx. Theory {\bfseries 83} (1995),
104--136.

\bibitem[HaT94a]{HaT94a} D.D Haroske, H. Triebel.
Entropy numbers in weighted function spaces and eigenvalue distributions of some degenerate 
pseudodifferential operators I. Math. Nachr. {\bfseries 167} (1994), 131--156.

\bibitem[HaT94b]{HaT94b} D.D Haroske, H. Triebel.
Entropy numbers in weighted function spaces and eigenvalue distributions of some degenerate 
pseudodifferential operators II. Math. Nachr. {\bfseries 168} (1994), 109--137.

\bibitem[HaT05]{HaT05} D.D Haroske, H. Triebel. Wavelet bases and entropy numbers in weighted function spaces. Math. Nachr. {\bfseries
278} (2005), 108--132.

\bibitem[Kab08]{Kab08} M. Kabanava. Tempered Radon measures. Rev. Mat. Complut. {\bfseries 21} (2008), 553--564.

\bibitem[Pie07]{Pie07} A. Pietsch. History of Banach spaces and linear operators. Birkh\"{a}user, Boston, 2007.

\bibitem[Rud91]{Rud91} W. Rudin. Functional analysis, sec. ed., McGraw-Hill, Inc., New York, 1991.

\bibitem[Ryd20]{Ryd20} E. Rydhe. On Laplace--Carleson embeddings, and $L^p$--mapping properties of the Fourier transform. Ark. Mat.
{\bfseries 58} (2020), 437--457.

\bibitem[Skr05]{Skr05} L. Skrzypczak. On approximation numbers of Sobolev embeddings of weighted function spaces. J. Approx. Theory
{\bfseries 136} (2005), 91--107.

\bibitem[T78]{T78} H. Triebel. Interpolation theory, function spaces, differential operators. North--Holland, Amsterdam, 1978. (Sec.
ed. Barth, Heidelberg, 1995).

\bibitem[T83]{T83} H. Triebel. Theory of function spaces. Birkh\"{a}user, Monographs Math. 78, Basel, 1983. 
   
\bibitem[T06]{T06} H. Triebel. Theory of function spaces III. Birkh\"{a}user, Monographs Math. 100, Basel, 2006. 

\bibitem[T08]{T08} H. Triebel. Function spaces and wavelets on domains. European Math. Soc. Publishing House, Z\"{u}rich, 2008.

\bibitem[T20]{T20} H. Triebel. Theory of function spaces IV. Birkh\"{a}user, Monographs Math. 107, Springer, Cham, 2020. 





\end{thebibliography}
\end{document}